\documentclass[12pt]{amsart}
\usepackage{amsfonts,latexsym,rawfonts,amsmath,amssymb,amsthm,mathrsfs}
\usepackage[plainpages=false]{hyperref}
\usepackage{graphicx}
\usepackage{comment}
\usepackage[pagewise]{lineno}

\numberwithin{equation}{section}

\RequirePackage{color}
\theoremstyle{plain}

\newtheorem{theorem}{Theorem}[section]
\newtheorem{lem}[theorem]{Lemma}

\theoremstyle{Corollary}
\newtheorem{cor}[theorem]{Corollary}

\newtheorem{remark}{Remark}[section]
\newcommand{\beq}{\begin{equation}}
\newcommand{\eeq}{\end{equation}}
\newcommand{\beqs}{\begin{eqnarray*}}
\newcommand{\eeqs}{\end{eqnarray*}}
\newcommand{\beqn}{\begin{eqnarray}}
\newcommand{\eeqn}{\end{eqnarray}}
\newcommand{\beqa}{\begin{array}}
\newcommand{\eeqa}{\end{array}}

\def\R{\mathbb R}

\begin{document}
\title{A class of anisotropic expanding curvature flows}

\author{Weimin Sheng}
\address{Weimin Sheng: School of Mathematical Sciences, Zhejiang University, Hangzhou 310027, China.}
\email{weimins@zju.edu.cn}

\author{Caihong Yi}
\address{Caihong Yi:
School of Mathematical Sciences, Zhejiang University, Hangzhou 310027, China.}
\email{11735001@zju.edu.cn}

\thanks{The authors were supported by NSFC, grant no. 11571304.}


\begin{abstract}
In this paper, we consider an expanding flow of smooth, closed, uniformly convex hypersurfaces in Euclidean $R^{n+1}$ with speed $u^\alpha\sigma_k^\beta$ firstly, where $u$ is support function of the hypersurface, $\alpha, \beta \in R^1$, and $\beta>0$, $\sigma_k$ is the $k$-th symmetric polynomial of the principle curvature radii of the hypersurface, $k$ is an integer and $1\le k\le n$. For $\alpha\le1-k\beta$, $\beta>\frac{1}{k}$ we prove that the flow has a unique smooth and uniformly convex solution for all time, and converges smoothly after normalisation, to a sphere centered at the origin. Moreover, for $\alpha\le1-k\beta$, $\beta>\frac{1}{k}$, we prove that the flow with the speed $fu^\alpha\sigma_k^\beta$ exists for all time and converges smoothly after normalisation to a soliton which is a solution of
$fu^{\alpha-1}\sigma_k^{\beta}=c$ provided that $f$ is a smooth positive function on $S^n$ and satisfies that $(\nabla_i\nabla_jf^{\frac{1}{1+k\beta-\alpha}}+\delta_{ij}f^{\frac{1}{1+k\beta-\alpha}})$ is positive definite. When $\beta=1$, our argument provides a proof to the well-known $L_p$ Christoffel-Minkowski problem for the case $p\ge k+1$ where $p=2-\alpha$, which is identify with Ivaki's recent result.
Especially, we obtain the same result for $k=n$ without any constraint on smooth positive function $f$. Finally,  we also give a counterexample for the two anisotropic expanding flows when $\alpha>1-k\beta$.
\end{abstract}

\maketitle

\baselineskip16pt
\parskip3pt

\section{Introduction}

 Flows of convex hypersurfaces in $\R^{n+1}$ by a class of speed functions which are homogeneous and symmetric in principal curvatures have been studied by many authors. Firey \cite{Fir74} first  introduced the Gauss curvature flow as a model for the shape change of tumbling stones. In \cite{Hui84} Huisken considered the mean curvature flow. He showed that the flow has a unique smooth solution and the hypersurface converges to a round sphere if the initial hypersurface is closed and convex. Later, a range of flows with the speed of homogeneous of degree one in principal curvatures were established, see \cite{Chow85, Chow87, And99} etc. for example.  For the curvature flow at the speed of $\alpha$-power of the Gauss-Knonecker curvature, it was conjectured that the solution will converge to a round point along the flow for $\alpha>\frac{1}{n+2}$. Chow \cite{Chow85}, Andrews \cite{And94}, Choi and Daskalopoulos \cite{ChoiDask16} gave the partial answers respectively. In \cite{BCD16}, Brendle et al.  finally resolved the conjecture for all $\alpha>\frac{1}{n+2}$ in all dimensions recently.

For the problem on the existence of the prescribed polynomial of the principal curvature radii of the hypersurface, Urbas\cite{Urb91}, Chow and Tsai\cite{ChowH97}, Gerhardt \cite{Gerh14}, Xia \cite{X16} studied the convergence for the flow with the speed of $F(\lambda_1,...\lambda_n)$, where $F$ is symmetric polynomial of the principal curvature radii $\lambda_i$ of the hypersurface. Guan and Ma \cite{GuMa03}, Hu et al.  \cite{HMS04}, and Guan-Xia \cite{GuX18} etc. gave the proofs for a class of $L_p$ Christoffel-Minkowski problems. On the other hand, as a nature extension, anisotropic flows usually provide alternative proofs and smooth category approach of the existence of solutions to elliptic PDEs arising in convex body geometry, see \cite{Wang96, ChWang00, GuNi16, LSW16, Iva2018} etc..
In this paper, we consider two expanding flows of the convex hypersurfaces at the speeds of $u^\alpha \sigma_k^\beta(\lambda_1,...\lambda_n)$ and $fu^\alpha\sigma_k^\beta(\lambda_1,...\lambda_n)$ respectively, where $u$ is the support function, $f$ is a smooth positive function on $S^n$, $\alpha, \beta \in R^1$, $\beta>\frac{1}{k}$ and $\sigma_k(\lambda_1,...,\lambda_n)$ is  the $k$-th symmetric polynomial of the principal curvature radii of the hypersurface, $k$ is an integer and $1\le k\le n$.
Generally, for the flow with the speed of high powers of curvatures, it is required that the initial hypersurface is uniformly convex and satisfies a suitable pinching conditions, so as to preserve the uniformly convexity and converge to a sphere (\cite{Al-Sin10},\cite{AndMc12}). Here we prove the same result without any pinching conditions for the flow at the speed $u^\alpha\sigma_k^\beta$. Moreover, for $\alpha\le1-k\beta$, and $\beta>\frac{1}{k}$ we prove that the solution to the flow which is moving at the speed $fu^\alpha\sigma_k^\beta$ exists for all time and converges smoothly after normalisation to a soliton which is a solution of
$fu^{\alpha-1}\sigma_k^{\beta}=c$ if $f$ is a smooth positive function on $S^n$ and satisfies the condition that $(\nabla_i\nabla_jf^{\frac{1}{1+k\beta-\alpha}}+\delta_{ij}f^{\frac{1}{1+k\beta-\alpha}})$ is positive definite.
We have the same result for $k=n$ without any constraint on positive smooth function $f$, which recovers the result of Chou and Wang \cite{ChWang06}. When $\beta=1$, the flow has been studied by Ivaki recently \cite{Iva2018}. In this case the self-similar solution of the flow  $\partial_tu=fu^\alpha\sigma_k^{\beta}$ is the solution to $fu^{\alpha-1}\sigma_k=c$ which is just the $L_p$ Christoffel-Minkowski problem for $p=2-\alpha$.

Let $M_0$ be a smooth, closed and uniformly convex hypersurface in $R^{n+1}$, and $M_0$ encloses the origin. We study the following anisotropic expanding curvature flow
\begin{equation}\label{EF1}
\left\{
\begin{array}{ll}
\frac{\partial{X}}{\partial{t}}(x,t)&=f(\nu)< X,\nu>^{{\alpha}}{\sigma}_{k}^{\beta}(x,t)\nu,\\
X(x,0)&=X_0(x),
\end{array}
\right.
\end{equation}
where $\sigma_k$ is the $k$-th elementary symmetric function for principal curvature radii, i.e
$$\sigma_k(.,t)=\sum_{i_1<...<i_k}\lambda_{i_1}...\lambda_{i_k},$$
$\lambda_i$ is the principal curvature radii of hypersurface $M_t$, parametrized by $X(.,t):S^n\to R^{n+1}$, and $\nu$ is the unit outer normal at $X(.,t)$.

In this paper, we prove the following
\begin{theorem}\label{main1}
Let $M_0$ be a smooth, closed and uniformly convex hypersurface in $R^{n+1}$, $n\ge2$, enclosing the origin. If $f\equiv1$, $\alpha\in R$, $\beta>\frac{1}{k}$ and $\alpha\le1-k\beta$, $k$ is an integer and $1\le k\le n$.
Then the flow \eqref{EF1} has a unique smooth and uniformly convex solution $M_t$ for all time $t>0$. After a proper rescaling $X\to\phi^{-1}(t)X$, where
\begin{equation}\label{scalingfactor1}
\begin{array}{lcl}
\phi(t)=e^{\gamma t}, &{\rm{if}}  &{\alpha}=1-k\beta,\\
\phi(t)=(1+(1-k\beta-{\alpha})\gamma t)^{\frac{1}{1-k\beta-\alpha}},&{\rm{ if}} &{\alpha}\ne 1-k\beta.
\end{array}
\end{equation}
 and
 $$\gamma=\sigma_k^{\beta}(1,...,1)$$
the hypersurface $\widetilde{M_t}=\phi^{-1}M_t$ converges exponentially to a sphere centered at the origin in the $C^\infty$ topology.
\end{theorem}
We denote by $u=<X,\nu>$ the support function of $M_t$ at $X$.
When $f\equiv 1$, by a direct calculation, we obtain from the flow \eqref{EF1} that the support function $u$ satisfies\\
\begin{equation}\label{EF11}
\left\{
\begin{array}{ll}
\frac{\partial{u}}{\partial{t}}(x,t)&=u^{{\alpha}}{\sigma}_{k}^{\beta},\\
u(.,0)&=u_0.
\end{array}
\right.
\end{equation}
Rescaling the hypersurface $M_t$ in the way of Theorem \ref{main1}, employing a new parameter $\tau$, we get the normalised flow
\begin{equation}\label{NEF1}
\left\{
\begin{array}{ll}
\frac{\partial{u}}{\partial{\tau}}(x,t)&=u^{{\alpha}}{\sigma}_{k}^{\beta}-\gamma u,\\
u(.,0)&=u_0.
\end{array}
\right.
\end{equation}
where \[\gamma=\sigma_k^{\beta}(1,...,1),\]
\begin{equation*}
\tau=\left\{
\begin{array}{lcl}
t &{\rm{if}}  &{\alpha}=1-k\beta,\\
\dfrac{\log[(1-\alpha-\beta k)\gamma t+1]}{1-\alpha-\beta k}&{\rm{ if}} &{\alpha}\ne 1-k\beta,
\end{array}
\right.
\end{equation*}

For general function $f$, under the flow \eqref{EF1}, the support function the support function $u$ satisfies \\
\begin{equation}\label{EF21}
\left\{
\begin{array}{ll}
\frac{\partial{u}}{\partial{t}}(x,t)&=fu^{{\alpha}}{\sigma}_{k}^{\beta},\\
u(.,0)&=u_0.
\end{array}
\right.
\end{equation}
We define
\[
\widetilde{u}=\Big(\frac{|S^n|}{V_{k+1}(\underbrace{u,u,...,u}_{(k+1)-\text{times}})}\Big)^{\frac{1}{k+1}}u,
\]
where the definition of $V_{k+1}(u,u,...,u)$ see Section 2.
Considering the normalised flow of \eqref{EF21} given by
\begin{equation}\label{NEF2}
\left\{
\begin{array}{ll}
\partial_\tau u&=fu^\alpha\sigma_k^{\beta}-u\frac{\int_{S^n}fu^\alpha\sigma_k^{1+\beta}d \mu_{S^n}}{|S^n|},\\
u(.,0)&=u_0.
\end{array}
\right.
\end{equation}
where\[\tau=\int_0^t\Big(\frac{|S^n|}{V_{k+1}(u,u,...,u)}\Big)^{\frac{1-k\beta-\alpha}{k+1}}ds.\]
In \eqref{NEF2}, we still use $u$ instead of $\tilde{u}$.
Consider the following functional which was introduced by Andrews in \cite{And97}
\begin{equation}\label{functional}
\mathcal{J}(u)=\int_{S^n}u\sigma_k(fu^{\alpha-1}\sigma_k^{\beta})^{-\frac{1}{\beta}}d\mu.
\end{equation}
When $\alpha\ge 1-k\beta$, we will prove in Lemma \ref{monotone} that $\mathcal{J}(u)$ is strictly decreasing along the flow\eqref{NEF2} unless $u$ solves the elliptic equation
\begin{equation}\label{elliptic eq}
fu^{\alpha-1}\sigma_k^{\beta}=c,
\end{equation}
where $c$ is a positive constant.

\begin{theorem}\label{main2}
Let $M_0$ be a smooth, closed and uniformly convex hypersurface in $R^{n+1}$, $n\ge2$, enclosing the origin. Suppose $\alpha, \beta\in R^1$, $\beta>\frac{1}{k}$ and $\alpha\le1-k\beta$, $k$ is an integer and $1\le k<n$, $f$ is a smooth positive function on $S^n$ and $\nabla_i\nabla_jf^{\frac{1}{1+k\beta-\alpha}}+\delta_{ij}f^{\frac{1}{1+k\beta-\alpha}}$ is positive definite.
Then the flow \eqref{EF1} has a unique smooth and uniformly convex solution $M_t$ for all time $t>0$. After normalisation, the rescaled hypersurfaces $\widetilde{M_t}$ converge smoothly to a smooth solution of \eqref{elliptic eq}, which is a minimiser of the functional \eqref{functional}.
\end{theorem}

For $k=n$, we obtain the same result without constrait on $f$ as follows.
\begin{theorem}\label{main3}
 Let $M_0$ be a smooth, closed and uniformly convex hypersurface in $R^{n+1}$, $n\ge2$, enclosing the origin. Suppose $\alpha \in R^1$, $\beta>\frac{1}{n}$ and $\alpha\le1-n\beta$.
Then for any $f$  smooth positive function on $S^n$, the flow \eqref{EF1} has a unique smooth and uniformly convex solution $M_t$ for all time $t>0$. After normalisation, the rescaled hypersurfaces $\widetilde{M_t}$ converge smoothly to a smooth solution of \eqref{elliptic eq}, which is a minimiser of the functional \eqref{functional}.
\end{theorem}

\begin{remark}
When $\beta=1$, our second flow \eqref{EF1} is just the one that Ivaki has studied recently \cite{Iva2018}. In \cite{Iva2018}, Ivaki employed the functional in \cite{GuX18} to prove that the flow  $\partial_tu=fu^\alpha\sigma_k$ has a unique smooth solution, and the rescaled flow converges smoothly to a homothetic self-similar solution which is a solution
 $fu^{\alpha-1}\sigma_k=c$ for $k<n$, $\alpha\le1-k$ and positive function $f$ satisfies that $(\nabla_i\nabla_jf^{\frac{1}{1+k-\alpha}}+\delta_{ij}f^{\frac{1}{1+k-\alpha}})$ is positive definite. When $k=n$, he obtained the same result without imposing any condition on $f$. It has been obtained by Chou and Wang in \cite{ChWang06}. Before Ivaki \cite{Iva2018}, Hu et al. \cite{HMS04} proved the existence result for the $L_p$ Christoffel Minkowski problem for $p=2-\alpha$: $fu^{\alpha-1}\sigma_k=1$,  $\alpha\le1-k$, $1\le k\le n-1$  for any positive function $f$ satisfying $(\nabla_i\nabla_jf^{\frac{1}{1+k-\alpha}}+\delta_{ij}f^{\frac{1}{1+k-\alpha}})\ge0$.
 Under the condition that $f$ is positive even fuction and satisfies $(\nabla_i\nabla_jf^{\frac{1}{1+k-\alpha}}+\delta_{ij}f^{\frac{1}{1+k-\alpha}})\ge0$, Guan and Xia in \cite{GuX18} obtained the existence result for $1-k<\alpha<1$.
 \end{remark}

We still denote by $u=<X,\nu>$ the support function of $M_t$ at $X$. We show the condition $\alpha\le1-k\beta$ is necessary. In fact, by use of the method of \cite{LSW16, LSW17}, we  show
\begin{theorem}\label{main4}
Suppose $\alpha>1-k\beta$, $\alpha, \beta\in \R$ and $\beta>0$, $k$ is an integer and $1\le k\le n$. There exist a smooth, closed, uniformly convex hypersurface $M_0$, such that under the flow \eqref{EF1},
\begin{equation}\label{ratio}
R(X(\cdot, t)):=\frac{\max_{S^n}u(\cdot, t)}{\min_{S^n}u(\cdot, t)}\to\infty,\qquad as\qquad t\to T
\end{equation}
for some $T>0$.
\end{theorem}

This paper is organised as follows. In Section 2, we recall some properties of convex hypersurfaces and show that the functional \eqref{functional} is strictly decreasing along the normalised flow \eqref{NEF2} unless $u$ satisfies the elliptic equation \eqref{elliptic eq}. In Section 3, we establish the a priori estimates, which ensure the long time existence of the normalized flows. In Section 4, we show that the flow \eqref{EF1} converge to the unit sphere (i.e. Theorem \ref{main1}) and complete the proofs of Theorem \ref{main2} and Theorem \ref{main3}. Finally, in Section 5, we prove Theorem \ref{main4}.

\section{Preliminary}
We recall some basic notations at first.
Let $M$ be a smooth, closed, uniformly convex hypersurface in $R^{n+1}$.  Assume that $M$ is parametrized by the inverse Gauss map $X:S^n\to M\subset R^{n+1}$ and encloses origin.
The support function $u:S^n\to R^1$ of $\mathit{M}$ is defined by
$$u(x)=\sup_{y\in \mathit{M}}\langle x,y\rangle.$$
The supremum is attained at a point $y=X(x)$, $x$ is the outer normal of $M$ at $y$. Hence
\[
u(x)=\langle x,X(x)\rangle.
\]

Let $e_1, \cdots, e_n$ be a smooth local orthonormal frame field on $S^n$, and ${\nabla}$ the covariant derivative with respect to the standard metric $e_{ij}$ on $S^n$. Denote by $g_{ij}$, $g^{ij}$, $h_{ij}$ the metric, the inverse of the metric and the second fundamental form of $M$, respectively. Then the second fundamental form of $M$ is given by (see e. g. \cite{Urb91})
\[
h_{ij}=\nabla_i\nabla_ju+ue_{ij},
\]and $h_{ij}$ is symmetric and satisfies the Codazzi equation
\[\nabla_ih_{jk}=\nabla_jh_{ik}.\]
By the Gauss-Weingarten formula
$$\nabla_ix=h_{jk}g^{kl}\nabla_lX,$$
we get
$$e_{ij}=\langle\nabla_ix,\nabla_jx\rangle=h_{ik}g^{kl}h_{jm}g^{ms}\langle\nabla_lX,\nabla_sX\rangle=g^{kl}h_{ik}h_{jl}.$$
Since $M$ is uniformly convex, $h_{ij}$ is invertible. Hence the principal curvature radii are the eigenvalues of the matrix
 \[
 b_{ij}=h^{ik}g_{jk}=\nabla_{ij}^2u+u\delta_{ij}.
 \]

Let $\sigma_k(A)$ be the $k$-th elementary symmetric function defined on the set $\mathcal{M}_n$ of $n\times n$ matrices and $\sigma_k(A_1,\cdots,A_k)$ be the complete polarization of $\sigma_k$ for $A_i\in\mathcal{M}_n$, $i=1,\cdots,k$,
i.e.
\[\sigma_k(A_1,\cdots,A_k)=\frac{1}{k!}\sum_{{i_1,...,i_k=1}
; {j_1,...,j_k=1}}^n\delta_{j_1,...,j_k}^{i_1,...,i_k}{({A}_{1}})_{i_1j_1}\cdots ({A_{k}})_{i_kj_k}.\]
Let $\Gamma_k$ be Garding's cone
\[\Gamma_k=\{A\in\mathcal{M}_n:\sigma_i(A)>0, i=1,...,k\}.\]
For a function $u\in C^2(S^n)$, we denote by $W_u$ the matrix
\[W_u:=\nabla^2u+uI.\]
In the case $W_u$ is positive definite, the eigenvalue of $W_u$ is the principal radii of a strictly convex hypersurface with support function $u$.
Let $u^i\in C^2(S^n)$, $i=1,\cdots,n+1$. Set
\[
V(u^1,u^2,\cdots,u^{n+1}):=\int_{S^n}u^1\sigma_n[W_{u^2},\cdots,W_{u^{n+1}}]d\mu,
\]
\[
V_{k+1}(u^1,u^2,\cdots,u^{k+1}):=V(u^1,u^2,\cdots,u^{k+1},1,\cdots,1).
\]
Therefore we define the $k+1$-th volume by
\[V_{k+1}(\underbrace{u,\cdots,u}_{(k+1)-\text{ times}}):=\int_{S^n} u\sigma_k[W_u,\cdots,W_u]d\mu.\]
Next, we state the well-known Alexandrov-Fenchel inequality.
\begin{lem}\label{AFI}(\cite{Sch93})
Let $u^i\in C^2(S^n)$, $i=1,2,\cdots,k$ be such that $u^i>0$ and $W_{u^i}\in\Gamma_k$ for $i=1,2,\cdots,k$. Then for any $v\in C^2(S^n)$, the Alexandrov-Fenchel inequality holds:
\[V_{k+1}(v,u^1,\cdots,u^k)^2\ge V_{k+1}(v,v,u^2,\cdots,u^k)V_{k+1}(u^1,u^1,u^2,\cdots,u^k),\]
the equality holds if and only if $v=au^1+\sum_{l=1}^{n+1}a_lx_l$ for some constants $a,a_1,\cdots,a_{n+1}$.
\end{lem}
We consider the flow \eqref{NEF2}. For convenience we still use $t$ instead of $\tau$ to denote the time variable
if no confusions arise, and we set
\begin{center}
$\rho=fu^{\alpha-1}\sigma_k^{\beta}$,
\qquad$\sigma[f]=\sigma_k[W_f,W_u,...,W_u]$\\
$\mathcal{Z}_p(u)=\int_{S^n}u\sigma_k(fu^{\alpha-1}\sigma_k^{\beta})^pd\mu=\int_{S^n}u\sigma_k\rho^pd\mu$,
\qquad$\eta(t)=\frac{\int_{S^n}fu^\alpha\sigma_k^{1+\beta}d \mu}{|S^n|}=\frac{\mathcal{Z}_1}{|S^n|}$,
\end{center}
where $p\in R^1$. We mention the fact that $\int_{S^n} u\sigma_k d\mu=|S^n|$ here which comes from the scaling of $\tilde u$.
Hence \eqref{NEF2} can be written as $\frac{\partial u}{\partial t}=\rho u-u\frac{\mathcal{Z}_1}{|S^n|}$.
By a direct calculation, we have
\begin{eqnarray}
&&\partial_t\mathcal{Z}_p(u)\nonumber\\
&=&\int_{S^n}(\rho u-u\frac{\mathcal{Z}_1}{|S^n|})\sigma_k\rho^{p}d\mu+\int_{S^n}ku\sigma[\rho u-u\frac{\mathcal{Z}_1}{|S^n|}]\rho^{p}d\mu
\nonumber\\
&&+\int_{S^n}p\rho^{p-1}u\sigma_k\Big((\alpha-1)fu^{\alpha-2}(\rho u-u\frac{\mathcal{Z}_1}{|S^n|})\sigma_k^{\beta}
+k\beta fu^{\alpha-1}\sigma_k^{\beta-1}\sigma[\rho u-u\frac{\mathcal{Z}_1}{|S^n|}]\Big)d\mu
\nonumber\\
&=&\mathcal{Z}_{1+p}-\frac{\mathcal{Z}_1\mathcal{Z}_{p}}{|S^n|}+k\int_{S^n}u\sigma[\rho u]\rho^{p}d\mu_{S^n}-k\frac{\mathcal{Z}_1\mathcal{Z}_{p}}{|S^n|}
\nonumber\\
&&+p(\alpha-1)\mathcal{Z}_{1+p}-p(\alpha-1)\frac{\mathcal{Z}_1\mathcal{Z}_{p}}{|S^n|}+k\beta p\int_{S^n}u\sigma[\rho u]\rho^{p}d\mu-k\beta p\frac{\mathcal{Z}_1\mathcal{Z}_{p}}{|S^n|}
\nonumber\\
&=&\big(1+p(\alpha-1)\big)\Big(\mathcal{Z}_{1+p}-\frac{\mathcal{Z}_1\mathcal{Z}_p}{|S^n|}\Big) +k(1+\beta p)\Big(\int_{S^n}u\rho^p\sigma[\rho u]d\mu-\frac{\mathcal{Z}_1\mathcal{Z}_p}{|S^n|}\Big).
\nonumber\end{eqnarray}

\begin{lem}\label{boundeta}
Suppose $\alpha\le1-k\beta$, $\beta>\frac{1}{k}$, $\eta(t)$ is uniformly bounded.
\end{lem}
\begin{proof}
Since $\alpha\le1-k\beta$, $\beta>\frac{1}{k}$, let $p=1$, we have
\begin{eqnarray*}
\partial_t\mathcal{Z}_1&=&\big(1+(\alpha-1)\big)\Big(\mathcal{Z}_{2}-\frac{\mathcal{Z}_1\mathcal{Z}_1}{|S^n|}\Big) +k(1+\beta)\Big(\int_{S^n}u\rho\sigma[\rho u]d\mu-\frac{\mathcal{Z}_1\mathcal{Z}_1}{|S^n|}\Big).
\end{eqnarray*}
Since $h_{ij}$ satisfies the Codazzi equations,
 we have $\sum_i\nabla_i\sigma^{ij}=0$ (see \cite{And94(1)},\cite{And97}), and
\begin{eqnarray*}
\int_{S^n}u\rho\sigma_k[W_{\rho u},W_u,...,W_u]d\mu
&=&\int_{S^n}u\rho\sigma[\rho u]d\mu\\
&=&\int_{S^n}u\rho\sigma^{ij}\big(\nabla_i\nabla_j(u\rho)+\delta_{ij} u\rho\big)d\mu
\\&=&\int_{S^n}u\rho\sigma^{ij}(h_{ij}\rho+2\nabla_iu\nabla_j\rho+u\nabla_i\nabla_j\rho)d\mu
\\&=&\mathcal{Z}_{2}-\int_{S^n}u^2\sigma^{ij}\nabla_i\rho \nabla_j\rho d\mu.
\end{eqnarray*}
By the Alexandrov-Fenchel inequality in Lemma \ref{AFI}, we have
\begin{eqnarray}\label{AF}
\Big(\int_{S^n}u\psi\sigma[u]d\mu\Big)^2&\ge&\int_{S^n}u\sigma[u]d\mu\int_{S^n}u\psi\sigma[u\psi]d\mu
\nonumber\\&=&|S^n|\Big(\mathcal{Z}_2-\int_{S^n}u^2\sigma^{ij}\nabla_i\psi\nabla_j\psi d\mu\Big).
\end{eqnarray}
Set $\psi=\rho$, the above inequality shows
\[\int_{S^n}u^2\sigma^{ij}\nabla_i\rho\nabla_j\rho d\mu-\mathcal{Z}_{2}+\frac{\mathcal{Z}^2_1}{|S^n|}\ge0.\]
Hence
\begin{eqnarray*}
\partial_t\mathcal{Z}_1
&=&\alpha\big(\mathcal{Z}_2-\frac{\mathcal{Z}^2_1}{|S^n|}\big)-k(1+\beta)\big(\int_{S^n}u^2\sigma^{ij}\nabla_i\rho\nabla_j\rho d\mu-\mathcal{Z}_{2}+\frac{\mathcal{Z}^2_{1}}{|S^n|}\big)
\\&\le&0,
\end{eqnarray*}
where $\alpha\le1-k\beta<0$ and the H\"{o}lder inequality shows that $\mathcal{Z}_{2}\ge\frac{\mathcal{Z}^2_{1}}{|S^n|}$.
\[\partial_t\eta=\frac{\partial_t\mathcal{Z}_1}{|S^n|}=\frac{1}{|S^n|}\Big(\alpha\big(\mathcal{Z}_2-\frac{\mathcal{Z}^2_1}{|S^n|}\big)-k(1+\beta)\big(\int_{S^n}u^2\sigma^{ij}\nabla_i\rho\nabla_j\rho d\mu-\mathcal{Z}_{2}+\frac{\mathcal{Z}^2_{1}}{|S^n|}\big)\Big)\le0.\]
Hence we obtain the uniform upper bound on $\eta(t)$.
Next we prove the uniform lower bound.
Set $\theta\le-\frac{1}{\beta}<0$, $\alpha\le1-k\beta$,  we have
\begin{eqnarray}
\partial_t\mathcal{Z}_\theta(u)\nonumber&=& \big(1+(\alpha-1)\theta\big)\Big(\mathcal{Z}_{1+\theta}-\frac{\mathcal{Z}_1\mathcal{Z}_\theta}{|S^n|}\Big) +k(1+\beta \theta)\Big(\int_{S^n}u\rho^\theta\sigma[\rho u]d\mu-\frac{\mathcal{Z}_1\mathcal{Z}_\theta}{|S^n|}\Big)
\nonumber\\&=&[1+(\alpha-1)\theta]\Big(\mathcal{Z}_{\theta+1}- \frac{\mathcal{Z}_1\mathcal{Z}_\theta}{|S^n|}\Big)
\nonumber\\&{}&+k(1+\theta\beta)\Big(\mathcal{Z}_{1+\theta}-\frac{4\theta}{(1+\theta)^2}\int_{S^n}u^2\sigma^{ij}\nabla_i(\rho^{\frac{1+\theta}{2}})\nabla_j(\rho^{\frac{1+\theta}{2}})- \frac{\mathcal{Z}_1\mathcal{Z}_\theta}{|S^n|}\Big).
\nonumber\end{eqnarray}
Set $\psi=\rho^{\frac{1+\theta}{2}}$ in the Alexandrov-Fenchel inequality \eqref{AF}, we obtain
\[\int_{S^n}u^2\sigma^{ij}\nabla_i\rho^{\frac{1+\theta}{2}}\nabla_j\rho^{\frac{1+\theta}{2}} d\mu-\mathcal{Z}_{1+\theta}+\frac{\mathcal{Z}^2_{\frac{1+\theta}{2}}}{|S^n|}\ge0.\]
Thus
\begin{eqnarray*}
\partial_t\mathcal{Z}_\theta(u)&=&[1+(\alpha-1)\theta+k(1+\theta\beta)]\Big(\mathcal{Z}_{\theta+1}- \frac{\mathcal{Z}_1\mathcal{Z}_\theta}{|S^n|}\Big)
-\frac{4\theta k(1+\theta\beta)}{(1+\theta)^2}\Big(\mathcal{Z}_{1+\theta}-\frac{\mathcal{Z}^2_{\frac{1+\theta}{2}}}{|S^n|}\Big)
\\&{}&-\frac{4\theta k(1+\theta\beta)}{(1+\theta)^2}\Big(\int_{S^n}u^2\sigma^{ij}\nabla_i\rho^{\frac{1+\theta}{2}}\nabla_j\rho^{\frac{1+\theta}{2}} d\mu-\mathcal{Z}_{1+\theta}+\frac{\mathcal{Z}^2_{\frac{1+\theta}{2}}}{|S^n|}\Big)
\\&\le&0
\end{eqnarray*}
since $\theta<-\frac{1}{\beta}$, and by the H\"{o}lder inequality, we get
$\mathcal{Z}_{1+\theta}\le\frac{\mathcal{Z}_1\mathcal{Z}_{\theta}}{|S^n|}$ and
$\mathcal{Z}_{1+\theta}\ge\frac{\mathcal{Z}^2_\frac{1+\theta}{2}}{|S^n|}$.
Hence, $\mathcal{Z}_\theta(u)\le\mathcal{Z}_\theta(u_0)$.
By the H\"{o}lder inequality again, we have
\[
|S^n|=\int_{S^n}u\sigma_kd\mu\le\big(\int_{S^n}fu^\alpha\sigma_k^{1+\beta}d\mu\big)^{\frac{-\theta}{1-\theta}}\big(\int_{S^n}u\sigma_k(fu^{\alpha-1}\sigma_k^{\beta})^\theta d\mu\big)^{\frac{1}{1-\theta}}=\mathcal{Z}^{\frac{-\theta}{1-\theta}}_1\mathcal{Z}^{\frac{1}{1-\theta}}_\theta.
\]
Therefore we get the uniform bound on $\eta(t)$.
\end{proof}

\begin{lem}\label{monotone}
The functional \eqref{functional} is non-increasing along the normalised flow \eqref{NEF2}, and the equality holds if and only if $M_t$ satisfies the elliptic equation \eqref{elliptic eq}.
\end{lem}
\begin{proof}
Since $\alpha\le1-k\beta$, $\beta>\frac{1}{k}$, from the above calculation process, when $p=-\frac{1}{\beta}$, we have along the normalised flow \eqref{NEF2}
\begin{eqnarray*}
\partial_t\mathcal{J}(u)=\partial_t\mathcal{Z}_{-\frac{1}{\beta}}(u)&=&
\frac{1-\alpha+\beta}{\beta}\Big(\mathcal{Z}_{1-\frac{1}{\beta}}-\frac{\mathcal{Z}_1\mathcal{Z}_{-\frac{1}{\beta}}}{|S^n|}\Big)
\\&\le&0.
\end{eqnarray*}
The last inequality holds from the H\"{o}lder inequality,
and the equality holds if and only if $fu^{\alpha-1}\sigma_k^{\beta}=c$, where $c$ is a constant.
\end{proof}

\section{A priori estimates}
We first show the $C^0$-estimate of the solution to \eqref{NEF1}.
\begin{lem}\label{C0estimate1}
Let $u(.,t)$, $t\in [0,T)$, be a smooth, uniformly convex solution to  \eqref{NEF1}. If $\alpha\le 1-k\beta$ and $\beta>\frac{1}{k}$, then there is a positive constant $C$ depending only on $\alpha$, $\beta$ and the lower and upper bounds of $u(\cdot, 0)$ such that
 \[
 \frac{1}{C}\le u(\cdot, t)\le C.
 \]
\end{lem}
\begin{proof}
Let $u_{\min}(t)=\min_{x\in S^n}u(\cdot, t)=u(x_t,t)$. For fixed time $t$, at the point $x_t$, we have
$$\nabla_iu=0\, \, \,   {\rm{and}}\, \, \,  \nabla_{ij}u\ge0,$$
then
$$\frac{du_{\min}}{dt}\ge\gamma u_{\min}(u_{\min}^{\alpha+k\beta-1}-1).$$
Hence
$u_{\min}\ge \min\{1,{u_{\min}(0)}\}$.
Similarly, we have $u_{\max}\le \max\{1, {u_{\max}(0)}\}$.
\end{proof}

\begin{lem}\label{bound usigmak2}
Let $n\ge2$, $1\le k\le n$, $\alpha\le1-k\beta$, $\beta>\frac{1}{k}$, and $X(.,t)$ be the solution to the normalised flow \eqref{NEF2} which encloses the origin for $t\in[0,T)$. Then there is a positive constant $C$ depending on the initial hypersurface and $f$, $\alpha, \beta$, such that
\[C^{-1}\le u^{\alpha-1}\sigma_k^\beta\le C .\]
\end{lem}
\begin{proof}
Consider the auxiliary function
\[Q=fu^{\alpha-1}\sigma_k^\beta.
\]
Since \[(fu^\alpha\sigma_k^\beta)_{ij}=Q_{ij}u+Q_iu_j+Q_ju_i+Qu_{ij},
\]
we get
\begin{eqnarray}\label{estimate1}
\nonumber\partial_tQ&=&(\alpha-1)f^2u^{2\alpha-2}\sigma_k^{2\beta}-\eta(\alpha-1+k\beta)fu^{\alpha-1}\sigma_k^\beta
\\
\nonumber&{}&+\beta f^2u^{2\alpha-1}\sigma_k^{2\beta-1}\sum_i^n\sigma_k^{ii}+\beta fu^{\alpha-1}u^{\alpha-2}\sigma_k^{\beta-1}\sigma_k^{ij}(fu^{\alpha}\sigma_k^\beta)_{ij}
\\
&=&(\alpha+k\beta-1)Q^2-\eta(\alpha+k\beta-1)Q+\beta fu^{\alpha}\sigma_k^{\beta-1}\sigma_k^{ij}Q_{ij}+2\beta fu^{\alpha-1}\sigma_k^{\beta-1}\sigma_k^{ij}Q_iu_j.
\end{eqnarray}
If $\alpha+k\beta-1\le0$, the sign of the coefficient of the highest order term $Q^2$ is negative. The sign of the coefficient of the lower order term $Q$ is positive. So it is easy to see $C^{-1}\le Q \le C$, where $C$ is the positive constant depending on  $f$, $\alpha, \beta, \min_{S^n\times [0,T)}u$ and $\max_{S^n\times [0,T)}u$.
\end{proof}
When $f=1$, by use of \eqref{estimate1}, we can get $C^{-1}\le Q \le C$ for the flow \eqref{NEF1}. Then by Lemma \ref{C0estimate1}, we have $C^{-1}\le \sigma_k \le C$. That is,
\begin{cor}\label{boundsigmak1}
Let $n\ge2$, $1\le k\le n$, $\alpha\le1-k\beta$, $\beta>\frac{1}{k}$ and $X(.,t)$ be the solution to the normalised flow \eqref{NEF1} which encloses the origin for $t\in[0,T)$. Then there is a positive constant $C$ depending on the initial hypersurface and $\alpha, \beta$, such that
\[C^{-1}\le\sigma_k\le C .\]
\end{cor}

\begin{lem}\label{C1estimate2}
Let $n\ge2$, $1\le k\le n$, $\alpha\le 1-k\beta$, $\beta>\frac{1}{k}$ and $X(.,t)$ be the solution to the normalised flow \eqref{NEF2} which encloses the origin for $t\in[0,T)$. Then there is a positive constant $C$ depending on the initial hypersurface and $f$, $\alpha, \beta$, such that $|\nabla\log u|\le C .$
\end{lem}
\begin{proof}
Let $w=\log u$. Then we have
\[
h_{ij}=u_{ij}+u\delta_{ij}=u(w_{ij}+w_iw_j+\delta_{ij})
\]
Assume the auxiliary function
 $|\nabla u|^2-Au^2<0$, for a positive constant $A>0$ along the flow. Otherwise there is a point $(x_{t_0}, t_0)$ where $t_0$is the first time, such that $|\nabla u|^2-Au^2=0$, $A>0$ is a constant to be decided later. Hence at the point $(x_{t_0},t_0)$, $\nabla_i|\nabla w|^2=0$.
 Choosing an orthonormal frame and rotating the the coordinate, such that $w_1=|\nabla w|$, $w_i=0$ for $i=2,\cdots,n$, and $(w_{ij})$ is diagonal. Then we get
\[
(a_{ij}):=(w_{ij}+w_iw_j+\delta_{ij})={\text{diag}} (1+w_1^2,1+w_{22},\cdots,1+w_{nn}),
\]
and
\begin{eqnarray*}
 0&\le&\partial_t(|\nabla u|^2-Au^2)=2u_i(fu^\alpha)_i\sigma_k^\beta+2\beta fu^\alpha\sigma_k^{\beta-1}\sigma_k^{mn}\nabla_ih_{mn}u_i-2Afu^{\alpha+1}\sigma_k^\beta
 \\&\le&2u_i(fu^\alpha)_i\sigma_k^\beta+4k\beta fu^{\alpha+1}\sigma_k^\beta-2\beta fu^{\alpha+2}\sigma_k^{\beta-1}\sum_i^n\sigma_k^{ii}-2\beta fu^\alpha\sigma_k^{\beta-1}\sigma_k^{ij}h_{li}h_{lj}
 \\&{}&+2A(k\beta-1)fu^{\alpha+1}\sigma_k^\beta-2A\beta fu^{\alpha+2}\sigma_k^{\beta-1}\sum_i^n\sigma_k^{ii}+2A\beta fu^{\alpha}\sigma_k^{\beta-1}\sigma_k^{ij}u_iu_j.
\end{eqnarray*}
Substituting $u_i=uw_i$ and  $w_1^2=A$ into the above inequality, and denote $\sigma_k=\sigma_k(a_{ij})$, we have
\begin{eqnarray*}
0\le A(\alpha+k\beta-1)+\sqrt{A}\frac{|\nabla f|}{f}+2k\beta-(A+1)^2\beta\frac{\sigma_k^{11}}{\sigma_k}-A\beta\frac{\sum_i^n\sigma_k^{ii}}{\sigma_k}+A^2\beta\frac{\sigma_k^{11}}{\sigma_k}
\end{eqnarray*}
Hence we have\\
Case 1. \qquad$\alpha\le1-k\beta$, then $A(1-k\beta-\alpha)\le\sqrt{A}\frac{|\nabla f|}{f}+2k\beta$.\\
Case 2. \qquad$\alpha=1-k\beta$, then $AC_0 \sigma_k^{-\frac{1}{k}}\le A\beta\frac{\sum_i^n\sigma_k^{ii}}{\sigma_k}\le\sqrt{A}\frac{|\nabla f|}{f}+2k\beta$.\\
Since $\frac{\sum_i^n\sigma_k^{ii}}{\sigma_k}\ge C(n,k)\sigma_k^{-\frac{1}{k}}$ by the classic Newton-MacLaurin inequality \cite{GL00}, and the fact that $\sigma_k(a_{ij})$ is bounded by Lemma \ref{bound usigmak2} for the Case 2.
Let $A$ be large enough we then get a contradiction. This completes the proof.
\end{proof}

 When $f=1$, by use of the same argument in Lemma \ref{C1estimate2} and the result of Lemma \ref{C0estimate1}. We have
\begin{cor}\label{C1estimate1}
 Let $n\ge2$, $1\le k\le n$, $\alpha\le 1-k\beta$, $\beta>\frac{1}{k}$ and $X(.,t)$ be the solution to the normalised flow \eqref{NEF1} which encloses the origin for $t\in[0,T)$. Then there is a positive constant $C$ depending on the initial hypersurface and $\alpha, \beta$, such that $|\nabla u|\le C .$
\end{cor}

\begin{lem}\label{C0estimate2 boundsigma2}
Let $u(.,t)$, $t\in [0,T)$, be a smooth, uniformly convex solution to \eqref{NEF2}. If $1\le k\le n$, $\alpha\le 1-k\beta$, $\beta>\frac{1}{k}$, then there is a positive constant $C_1$ and $C_2$ depending only on $f$, $\alpha$, $\beta$ and the lower and upper bounds of $u(\cdot, 0)$ such that
 \[
 \frac{1}{C_1}\le u(\cdot, t)\le C_1,\qquad \frac{1}{C_2}\le\sigma_k\le C_2.
 \]
\end{lem}
\begin{proof}
Since for the normalised flow \eqref{NEF2}, $\int_{S^n}u\sigma_kd\mu=|S^n|$ is constant. From Lemma \ref{bound usigmak2}, there is a positive constant $C$, such that $C^{-1}\le u^{\alpha-1}\sigma_k^{\beta}\le C$. Hence we have
\[C^{-\frac{1}{\beta}}u_{\min}^{\frac{1-\alpha+\beta}{\beta}}(t)\le\frac{\int_{S^n}u\sigma_kd\mu}{|S^n|}\le C^{\frac{1}{\beta}} u_{\max}^{\frac{1-\alpha+\beta}{\beta}}(t)\]
Hence we obtain the uniform lower and upper bounds on $u$ from Lemma \ref{C1estimate2}. Then by Lemma \ref{bound usigmak2}, we get the uniform lower and upper bounds on $\sigma_k$.
\end{proof}
Now we are going to estimate the upper and lower bounds of the principle curvature radii of the hypersurface $M_t$.
We rewrite the equation \eqref{NEF2} in the following form
$$\frac{\partial{u}}{\partial{t}}=fu^{{\alpha}}{\sigma}_{k}^\beta(u)-\eta u=\Phi L-\eta u,$$
where $L=F^{\beta k}$, $F=\sigma_k^{\frac{1}{k}}$, $\Phi=fu^\alpha$, and $h_{ij}=u_{ij}+u\delta_{ij}$.
\begin{lem}\label{boundsprincipleradii2}
Let $1\le k< n$, $\alpha, \beta\in R^1$, $\beta>\frac{1}{k}$ and $\alpha\le1-k\beta$, $X(., t)$ be the solution to the normalised flow \eqref{NEF2} for $t\in[0,T)$, which encloses the origin. Assume $f$ is a smooth positive function on $S^n$ and $(\nabla_i\nabla_jf^{\frac{1}{1+k\beta-\alpha}}+\delta_{ij}f^{\frac{1}{1+k\beta-\alpha}})$ is positive definite. Then there is a constant $C$ depending only on $f$, $\alpha, \beta, \min_{S^n\times[0,T)}u$ and $\max_{S^n\times [0,T)}u$, such that the principal curvature radii of $X(\cdot, t)$ are bounded from above and below
$$\frac{1}{C}\le\lambda_i(.,t)\le C$$ for all $t\in[0,T)$ and $i=1,...,n$.
\end{lem}
\begin{proof}
Suppose the maximum eigenvalue of the matrix $[\frac{h^{ij}}{u}]$  at time $t$ is attained at the point $x_t$ with unit eigenvector $\xi_t\in T_{x_t}S^n$. By a rotating the frame $e_1, \cdots, e_n$ at $x_t$, assume that at $x_t$ we have $\xi_t=e_1$. At $(x_t,t)$, we have
\begin{eqnarray*}
{\partial_t{h^{11}}}&=&{\Phi}L_{kl}^{'}{\nabla}_k{\nabla}_lh^{11}-{\Phi}({h^{11}})^2(L+L_{kl}^{'}h_{kl})+{\Phi}(trL^{'})h^{11}+\eta h^{11}\\
& &-{\Phi}({h^{11}})^2\big(2L_{km}^{'}h^{nl}+L_{kl,mn}^{''}\big){\nabla}_1h_{kl}{\nabla}_1h_{mn}
-({h^{11}})^2(2{\nabla}_1{\Phi}{\nabla}_1L+L{\nabla}_1{\nabla}_1{\Phi}).
\end{eqnarray*}
Since $F=\sigma_k^{\frac{1}{k}}$ is concave and homogeneous of degree one, from \cite{Urb91}
\begin{equation}
F_{ij}h_{ij}=F,
\end{equation}
\begin{equation}
(2F_{km}h^{nl}+F_{kl,mn}){\nabla}_1h_{kl}{\nabla}_1h_{mn}\ge2F^{-1}F_{kl}F_{mn}{\nabla}_1h_{kl}{\nabla}_1h_{mn},
\end{equation}
we have
\begin{eqnarray*}
{\partial_t{h^{11}}}&=&{\Phi}k{\beta}F^{k{\beta}-1}F_{kl}{\nabla}_k{\nabla}_lh^{11}-{\Phi}({h^{11}})^2(F^{k \beta}+k \beta F^{k \beta -1}F_{kl}h_{kl})+{\Phi}k \beta F^{k{\beta}-1}\sum_{i}^{n}F_{ii}h^{11}
\\&{}&-{\Phi}({h^{11}})^2\Big(2k\beta F^{k\beta-1}F_{km}h^{nl}+k\beta (k\beta-1)F^{k\beta-2}F_{mn}F_{kl}
\\&{}&+k\beta F^{k\beta-1}F_{kl,mn}\Big){\nabla}_1h_{kl}{\nabla}_1h_{mn}
-({h^{11}})^2(2{\nabla}_1{\Phi}k\beta F^{k\beta -1}{\nabla}_1F+F^{k\beta}{\nabla}_1{\nabla}_1{\Phi})+\eta h^{11}
\\&\le&{\Phi}k{\beta}F^{k{\beta}-1}F_{kl}{\nabla}_k{\nabla}_lh^{11}-{\Phi}(h^{11})^2(1+k\beta) F^{k\beta}+\Phi k\beta F^{k\beta-1}\sum_{i}^{n}F_{ii}h^{11}+\eta h^{11}
\\&{}&-{\Phi}(h^{11})^2 k{\beta}(k{\beta}+1)F^{k{\beta}-2}({\nabla}_{1}F)^{2}-(h^{11})^2(2k{\beta}F^{k{\beta}-1}{\nabla}_1F{\nabla}_1{\Phi}+F^{k{\beta}}{\nabla}_1{\nabla}_1{\Phi}).
\end{eqnarray*}
Since
\[
2k\beta F^{k\beta-1}\nabla_1F\nabla_1\Phi\le\Phi k\beta(k\beta+1)F^{k\beta}(\frac{\nabla_1F}{F})^2+\frac{k\beta}{k\beta+1}F^{k\beta}{\frac{({\nabla}_1{\Phi})^2}{{\Phi}}},
\]
we have
\begin{eqnarray*}
\partial_th^{11}&\le&{\Phi}k{\beta}F^{k{\beta}-1}F_{kl}{\nabla}_k{\nabla}_lh^{11}-{\Phi}F^{k\beta}(1+k\beta)(h^{11})^2+\Phi k\beta F^{k\beta-1}\sum_{i}^{n}F_{ii}h^{11}+\eta h^{11}
\\&{}&+\frac{k\beta}{k\beta+1}F^{k\beta}{\frac{({\nabla}_1{\Phi})^2}{{\Phi}}}(h^{11})^2-F^{k\beta}{\nabla}_1{\nabla}_1{\Phi}(h^{11})^2.
\end{eqnarray*}
Since $\nabla_i\frac{h^{11}}{u}=0$ and $\nabla_{ij}\frac{h^{11}}{u}\le0$,
\begin{eqnarray*}
\partial_t\frac{h^{11}}{u}&\le&-\Phi(1+k\beta)uF^{k\beta}(\frac{h^{11}}{u})^2+\frac{k\beta}{k\beta+1}uF^{k\beta}{\frac{({\nabla}_1{\Phi})^2}{{\Phi}}}(\frac{h^{11}}{u})^2
\\&{}&-uF^{k\beta}{\nabla}_1{\nabla}_1{\Phi}(\frac{h^{11}}{u})^2
+(k\beta-1)u^{-1}F^{k\beta}\Phi\frac{h^{11}}{u}+2\eta\frac{h^{11}}{u}
\\&\le&u\Phi F^{k\beta}(\frac{h^{11}}{u})^2\Big(\frac{k\beta}{1+k\beta}(\frac{f_1}{f})^2- \frac{f_{11}}{f}-\frac{2\alpha}{1+k\beta}\frac{u_1}{u}\frac{f_1}{f}+\frac{\alpha(1+k\beta-\alpha)}{1+k\beta} (\frac{u_1}{u})^2\Big)
\\&{}&-\Phi(1+k\beta-\alpha)uF^{k\beta}(\frac{h^{11}}{u})^2+(k\beta-1-\alpha)u^{-1}F^{k\beta}\Phi\frac{h^{11}}{u}+2\eta\frac{h^{11}}{u}
\end{eqnarray*}
Since $(f^{\frac{1}{1+k\beta-\alpha}})_{11}+f^{\frac{1}{1+k\beta-\alpha}}>0$, substituting it into the above inequality, we have
\begin{eqnarray*}
\partial_t\frac{h^{11}}{u}&\le&u\Phi F^{k\beta}(\frac{h^{11}}{u})^2\Big(-\frac{f_{11}}{f}+\frac{k\beta-\alpha}{1+k\beta-\alpha}(\frac{f_1}{f})^2-(1+k\beta-\alpha)
\\&{}&+\alpha\big(\sqrt{\frac{1+k\beta-\alpha}{1+k\beta}}\frac{u_1}{u}-\sqrt{\frac{1}{(1+k\beta)(1+k\beta-\alpha)}}\frac{f_1}{f}\big)^2 \Big)
\\&{}&+(k\beta-1-\alpha)u^{-1}F^{k\beta}\Phi\frac{h^{11}}{u}+2\eta\frac{h^{11}}{u}
\end{eqnarray*}
For $\beta>\frac{1}{k}$ and  $\alpha\le1-k\beta<0$, we get $k\beta-1-\alpha\ge2(k\beta-1)>0$. Hence
\[\partial_t\frac{h^{11}}{u}\le-C_1(\frac{h^{11}}{u})^2+C_2\frac{h^{11}}{u},\]
that is, $h^{11}\le C$, where $C$ depends on the initial hypersurface,  the minimum eigenvalue of $(\nabla_i\nabla_jf^{\frac{1}{1+k\beta-\alpha}}+\delta_{ij}f^{\frac{1}{1+k\beta-\alpha}})$, $f$, $\alpha$, and $\beta$.
Now together with  Lemma \ref{C0estimate2 boundsigma2},
 we get
\begin{equation}
       C^{-1}\le\lambda_i\le C.
\end{equation}
We therefore complete the proof.
\end{proof}

\begin{lem}\label{boundsprincipleradii1}
Let $1\le k\le n$, $\alpha, \beta\in R^1$, $\beta>\frac{1}{k}$ and $\alpha\le 1-k\beta$. Let $X(., t)$ be the solution to the normalised flow \eqref{NEF1} for $t\in[0,T)$, which encloses the origin. Then there is a constant $C$ depending only on the initial hypersurface and $\alpha, \beta$, such that the principal curvature radii of $X(\cdot, t)$ are bounded from above and below
$$\frac{1}{C}\le\lambda_i(.,t)\le C$$ for all $t\in[0,T)$ and $i=1,...,n$.
\end{lem}
\begin{proof}
We prove the lemma just as Lamma \ref{boundsprincipleradii2}.
Suppose the maximum eigenvalue of the matrix $[h^{ij}]$  at time $t$ is attained at the point $x_t$ with unit eigenvector $\xi_t\in T_{x_t}S^n$. By rotating the frame $e_1, \cdots, e_n$ at $x_t$, assume that at $x_t$ we have $\xi_t=e_1$. At $(x_t,t)$, we have
\begin{eqnarray*}
{\partial_t{h^{11}}}&=&{\Phi}k{\beta}F^{k{\beta}-1}F_{kl}{\nabla}_k{\nabla}_lh^{11}-{\Phi}({h^{11}})^2(F^{k \beta}+k \beta F^{k \beta -1}F_{kl}h_{kl})+{\Phi}k \beta F^{k{\beta}-1}\sum_{i}^{n}F_{ii}h^{11}
\\&{}&-{\Phi}({h^{11}})^2\Big(2k\beta F^{k\beta-1}F_{km}h^{nl}+k\beta (k\beta-1)F^{k\beta-2}F_{mn}F_{kl}
\\&{}&+k\beta F^{k\beta-1}F_{kl,mn}\Big){\nabla}_1h_{kl}{\nabla}_1h_{mn}
\\&{}&-({h^{11}})^2(2{\nabla}_1{\Phi}k\beta F^{k\beta -1}{\nabla}_1F+F^{k\beta}{\nabla}_1{\nabla}_1{\Phi})+\gamma h^{11}
\\&\le&-{\Phi}(h^{11})^2 F^{k\beta}+\gamma h^{11}-{\Phi}(h^{11})^2 k{\beta}(k{\beta}+1)F^{k{\beta}-2}({\nabla}_{1}F)^{2}
\\&{}&-(h^{11})^2(2k{\beta}F^{k{\beta}-1}{\nabla}_1F{\nabla}_1{\Phi}+F^{k{\beta}}{\nabla}_1{\nabla}_1{\Phi}).
\end{eqnarray*}
Since
\[
2k\beta F^{k\beta-1}\nabla_1F\nabla_1\Phi\le\Phi k\beta(k\beta+1)F^{k\beta}(\frac{\nabla_1F}{F})^2+\frac{k\beta}{k\beta+1}F^{k\beta}{\frac{({\nabla}_1{\Phi})^2}{{\Phi}}},
\]
we have
\begin{eqnarray}
\partial_th^{11}&\le&(\alpha-1)u^{\alpha}\sigma_k^{\beta}(h^{11})^2+(\gamma-\alpha u^{\alpha-1}{\sigma_k}^{\beta})h^{11}
\nonumber\\
&{}&+\frac{\alpha(1+ k\beta-\alpha)}{k\beta+1}u^{\alpha-2}({\nabla}_1u)^2\sigma_k^{\beta}{(h^{11})^2}. \label{h^{11}dert1}
\end{eqnarray}
Since $\beta>\frac{1}{k}$, $\alpha\le 1-k\beta<0$, by \eqref{h^{11}dert1} we have
\[
\partial_th^{11}\le -C_0(h^{11})^2+C_1h^{11}
\]
where $C_0$ and $C_1$ are two positive constants which depend only on the initial hypersurface and $\alpha, \beta$. Therefore we have
\[
h^{11}\le \max\{C_3, \max h^{11}(0)\}
\]
where  $C_3$ is also a positive constants depending only on the initial hypersurface and $\alpha, \beta$. Now together with Lemma \ref{boundsigmak1},
 we get
\begin{equation}
       C^{-1}\le\lambda_i\le C.
\end{equation}
\end{proof}

Now we show Lemma \ref{boundsprincipleradii2} holds for any positive smooth function $f$ when $k=n$, $\beta>\frac{1}{n}$ and $\alpha\le1-n\beta$.
\begin{lem}\label{boundsprincipleradii3}
Let $n\ge2$, $k=n$. If $\alpha\le1-n\beta$, $\beta>\frac{1}{n}$ and $X(\cdot, t)$ be the solution to the normalised flow \eqref{NEF2} which encloses the origin for $t\in [0,T)$. Then there is a constant $C$ depending only on the initial hypersurface and $f$, $\alpha$, $\beta$, such that the principal curvature radii of $X(\cdot, t)$ are bounded from above and below
\[C^{-1}\le\lambda_i(.,t)\le C\] for all $t\in[0,T)$ and $i=1,...,n$.
\end{lem}

\begin{proof}
Consider the auxiliary function
\[w(\xi, t, \tau)=\log h^{\tau\tau} -\varepsilon \log u+\frac{M}{2}({u}^2+|\nabla u|^2),\]
where
$\tau$ is a unit vector in the tangential space of $S^n$, while $\epsilon $ and $M$ are large constants to be decided. Assume $w$ achieve its maximum at $(x_0,t_0)$ in the direction $\tau=(1, 0, \cdots, 0)$. By a coordinate rotation, $h_{ij}$ and $h^{ij}$ are diagonal at this point. Then at the point $(x_0,t_0)$.
\[w=\log h^{11}-\varepsilon  \log u+\frac{M}{2}({u}^2+|\nabla u|^2),\]
\begin{eqnarray}\label{partial1}
0=\nabla_iw=-h^{11}\nabla_ih_{11}-\varepsilon \frac{\nabla_{i}u}{u}+Muu_{i}+M\nabla_{k}u\nabla_{ki}u,
\end{eqnarray}
\begin{eqnarray}\label{partial2}
0\ge\nabla_{ij}w&=&-h^{11}\nabla_{ij}h_{11}+2h^{11}h^{kk}\nabla_{1}h_{ik}\nabla_{1}h_{kj}-(h^{11})^2\nabla_ih_{11}\nabla_jh_{11}
-\varepsilon \frac{\nabla_{ij}u}{u}
\nonumber\\&{}&+\varepsilon \frac{\nabla_{i}u\nabla_{j}u}{{u}^2}+Mu_iu_j+Muu_{ij}+M\nabla_{ki}u\nabla_{kj}u+M\nabla_ku\nabla_{kij}u
\end{eqnarray}
Set $\Phi=fu^{\alpha}$, we have
\begin{eqnarray}
0\le\partial_tw&=&-h^{11}\partial_th_{11}-\varepsilon \frac{\partial_tu}{u}+Mu\nabla_{t}u+M\nabla_{k}u\nabla_{kt}u
\nonumber\\
&=&-h^{11}\Big(\Phi_{11}{\sigma_n}^\beta+2\beta{\sigma_n}^{\beta-1}\nabla_1\Phi\nabla_1\sigma_n+\beta(\beta-1)\Phi{\sigma_n}^{\beta-2}(\nabla_1\sigma_n)^2
\nonumber\\
&{}&+\beta\Phi{\sigma_n}^{\beta-1}\nabla_{11}\sigma_n
-\eta h_{11}+\Phi{\sigma_n}^{\beta}\Big)-\varepsilon \frac{\Phi {\sigma_n}^{\beta}-\eta u}{u}
\nonumber\\
&{}&+Mu(\Phi {\sigma_n}^{\beta}-\eta u)+Mu_k(\Phi {\sigma_n}^{\beta}-\eta u)_k
\nonumber\\
&=&-h^{11}\Big(\Phi_{11}{\sigma_n}^\beta+2\beta{\sigma_n}^{\beta}\nabla_1\Phi h^{ij}\nabla_1h_{ij}+\beta(\beta-1)\Phi{\sigma_n}^{\beta}(h^{ij}\nabla_1h_{ij})^2
\nonumber\\
&{}&+\beta\Phi{\sigma_n}^{\beta}\big(h^{ij}\nabla_{ij}h_{11}+n-h_{11}\sum_ih^{ii}-h^{ii}h^{jj}(\nabla_1h_{ij})^2+(h^{ij}\nabla_1h_{ij})^2\big)
\nonumber\\
&{}&-\eta h_{11}+\Phi{\sigma_n}^{\beta}\Big)
-\varepsilon \frac{\Phi {\sigma_n}^{\beta}-\eta u}{u}+Mu(\Phi {\sigma_n}^{\beta}-\eta u)+Mu_k(\Phi {\sigma_n}^{\beta}-\eta u)_k
\nonumber
\end{eqnarray}
By \eqref{partial2} and multiplying $\Phi^{-1}{\sigma_n}^{-\beta}$ the two sides of the above inequality,we obtain
\begin{eqnarray*}
0&\le&-h^{11}\frac{\nabla_{11}\Phi}{\Phi}+h^{11}\big((\frac{\nabla_1\Phi}{\Phi})^2+\beta^2(h^{ij}\nabla_1h_{ij})^2\big) -h^{11}\beta^2(h^{ij}\nabla_1h_{ij})^2
\nonumber\\
&{}&+\beta\Big(-2h^{11}h^{ij}h^{kk}\nabla_1h_{ik}\nabla_1h_{jk}
+(h^{11})^2h^{ij}\nabla_ih_{11}\nabla_jh_{11}
+n\varepsilon  {u}^{-1}
\nonumber\\
&{}&-\varepsilon \sum_ih^{ii}-\varepsilon  h^{ij}u_iu_j{u}^{-2}-Mh^{ij}h_{ik}h_{jk}+nMu
-Mh^{ij}u_k\nabla_kh_{ij}\Big)
\nonumber\\
&{}&-h^{11}(1+n\beta)+\beta\sum_ih^{ii}+\beta h^{11}h^{ii}h^{jj}(\nabla_1h_{ij})^2+Mu+Mu_k\frac{\nabla_k\Phi}{\Phi}
\nonumber\\&{}&+\beta M h^{ij}u_k\nabla_kh_{ij}
-\eta M\frac{{u}^2+|\nabla u|^2}{\Phi{\sigma_n}^\beta}-\varepsilon  {u}^{-1}+\eta\frac{\varepsilon +1}{\Phi{\sigma_n}^\beta}
\\&\le&-2\beta h^{11}h^{ij}h^{kk}\nabla_{1}h_{ik}\nabla_{1}h_{kj}+\beta(h^{11})^2h^{ij}\nabla_{i}h_{11}\nabla_{j}h_{11}+\beta h^{11}h^{ii}h^{jj}(\nabla_{1}h_{ij})^2
\nonumber\\
&{}&+\beta\frac{n\varepsilon }{u}
-\beta\varepsilon \sum h^{ii}
+\beta\sum_ih^{ii}+Mu+h^{11}(\frac{\nabla_1\Phi}{\Phi})^2
\nonumber\\
&{}&+M\beta nu-h^{11}\frac{\nabla_1\nabla_1\Phi}{\Phi}+ Mu_k\frac{\nabla_k\Phi}{\Phi}-\eta\frac{M{r}^2-\varepsilon -1}{\Phi {\sigma_n}^\beta}
\\&\le& C_0-\beta(\varepsilon -1) h^{11}+h^{11}(\frac{\nabla_1\Phi}{\Phi})^2-h^{11}\frac{\nabla_1\nabla_1\Phi}{\Phi} +Mu_k\frac{\nabla_k\Phi}{\Phi}-\eta\frac{M{r}^2-\varepsilon -1}{\Phi {\sigma_n}^\beta},
\end{eqnarray*}
where we use the Cauchy inequality $2\beta\frac{\nabla_1\Phi}{\Phi} h^{ij}\nabla_1h_{ij}\le(\frac{\nabla_1\Phi}{\Phi})^2+\beta^2(h^{ij}\nabla_1h_{ij})^2$ for the second term.\\
Since $\nabla_k\Phi=\nabla_k\big(fu^{\alpha}\big)$, we obtain
\begin{eqnarray*}
& &h^{11}(\frac{\nabla_1\Phi}{\Phi})^2-h^{11}\frac{\nabla_1\nabla_1\Phi}{\Phi} +Mu_k\frac{\nabla_k\Phi}{\Phi}\\
&=&h^{11}\left(\frac{f_1}{f}+\alpha\frac{u_1}{u}\right)^2-h^{11}\Big(\frac{\nabla_1\nabla_1f}{f}+2\alpha\frac{\nabla_1f}{f}\frac{\nabla_1u}{u}+\alpha(\alpha-1)(\frac{\nabla_1u}{u})^2\Big)
\\&{}&+M(\nabla_ku\frac{\nabla_kf}{f}+\alpha u^{-1}|\nabla u|^2) \\&\le&Ch^{11}+CM
\end{eqnarray*}
Choosing $M\ge\frac{\varepsilon+1}{\min {r^*}^2}$, the inequality becomes
\[0\le C_0-\beta(\varepsilon -1) h^{11}+Ch^{11}+CM.\]
By choosing $\varepsilon $ large to get
\[
0\le C_1-C_2h^{11}.
\]
That is, $h^{11}\le C_3$, where $C_3$ is a constant depending only on $f$, $\alpha$, $\beta$, $\min_{S^n\times[0,T)}u$ and $\max_{S^n\times [0,T)}u$. Hence the principal radii $\lambda_i\ge C_2, i=1, \cdots, n$.
From Lemma \ref{C0estimate2 boundsigma2}, we know $\sigma_n=\lambda_1\cdots\lambda_n\le C_1$.
Therefore we get the $C^2$ estimate $C^{-1}\le\lambda_i\le C, i=1,\cdots,n$ for the solutions to the normalised flow \eqref{NEF2}.
\end{proof}

From the estimates obtained in Lemmata \ref{boundsprincipleradii1}, \ref{boundsprincipleradii2} and  \ref{boundsprincipleradii3}, we know that the equations \eqref{NEF1} and \eqref{NEF2} are uniformly parabolic. By the $C^0$ estimates (Lemmata \ref{C0estimate1} and \ref{C0estimate2 boundsigma2}), the gradient estimates (Lemma \ref{C1estimate1} and Corollary \ref{C1estimate2}) and the $C^2$ estimates Lammata \ref{boundsprincipleradii1}, \ref{boundsprincipleradii2} and  \ref{boundsprincipleradii3}, and the Krylov's theory \cite{Knv87}, we get the H\"{o}lder continuity of $\nabla^2u$ and $u_t$. Then we can get higher order derivation estimates by the regularity theory of the uniformly parabolic equations. Therefore we get the long time existence and the uniqueness of the smooth solution to the normalized flows \eqref{NEF1} and \eqref{NEF2}, respectively.

\section{Proof of Theorem \ref{main1}, Theorem \ref{main2} and Theorem \ref{main3} }
In this section we give the proof of Theorems \ref{main1} at first.
In order to prove the convergence of the normalized flow \eqref{NEF1}, we require
the following better gradient estimate.
\begin{lem}\label{GradEstimate3}
Let $u(.,t)$ be a smooth uniformly convex solution to the flow \eqref{NEF1}. If $\alpha\le1-k\beta$, then there exist positive constants $C$ and $C_0$, depending only on the initial hypersurface and $\alpha$, $\beta$, such that
$$\max_{S^n}\frac{|\nabla u(\cdot, t)|}{u(\cdot, t)}\le Ce^{-C_0t}$$
for all $t>0$.
\end{lem}

\begin{proof}
Let $w=\log u$. Then we have
\[
h_{ij}=u_{ij}+u\delta_{ij}=e^w(w_{ij}+w_iw_j+\delta_{ij})
\]
and
\[
w_t=(e^w)^{\alpha+k\beta-1}\sigma_k^{\beta}([w_{ij}+w_iw_j+\delta_{ij}])-\gamma.
\]
Consider the auxiliary function
 $$Q=\frac{1}{2}|{\nabla}w|^2.$$
 At the point where $Q$ attains its spatial maximum, we have
\[
0={\nabla}_iQ=\sum w_lw_{li},
\]
\[
0\ge{\nabla}_{ij}Q=\sum w_{li}w_{lj}+\sum w_l w_{lij},
\]
and
\begin{eqnarray*}
{\partial}_tQ&=&\sum w_lw_{lt}
\\&=&2Q(\alpha+k\beta-1)(e^w)^{\alpha+k\beta-1}\sigma_k^{\beta}
\\&{}&+\beta(e^w)^{\alpha+k\beta-1}\sigma_k^{\beta-1}\sigma_k^{ij}(w_lw_{ijl}+w_lw_{il}w_j+w_lw_iw_{jl})
\\&=&2Q(\alpha+k\beta-1)(e^w)^{\alpha+k\beta-1}\sigma_k^{\beta}+\beta(e^w)^{\alpha+k\beta-1}\sigma_k^{\beta-1}\sigma_k^{ij}w_lw_{ijl}.
\end{eqnarray*}
By the Ricci identity,
$${\nabla}_lw_{ij}={\nabla}_jw_{li}+\delta_{il}w_j-\delta_{ij}w_l,$$
we get
\begin{eqnarray*}
{\partial}_tQ&=&2(\alpha+k\beta-1)(e^w)^{\alpha+k\beta-1}\sigma_k^{\beta}Q\\
& &+\beta(e^w)^{\alpha+k\beta-1}\sigma_k^{\beta-1}\sigma_k^{ij}(\sum _lw_lw_{lij}+w_iw_j-\delta_{ij}|{\nabla}w|^2)
 \\&\le&2(\alpha+k\beta-1)(e^w)^{\alpha+k\beta-1}\sigma_k^{\beta}Q+2\beta(e^w)^{\alpha+k\beta-1}\sigma_k^{\beta-1}(\max_i\sigma_k^{ii}-\sum_i\sigma_k^{ii})Q
\\&\le&-C_0Q.
\end{eqnarray*}
We therefore have $Q\le Ce^{-C_0t}$, where $C_0$ and $C$ are two positive constants which depends only on $\alpha$, $\beta$ and the geometry of $M_0$.
\end{proof}

\begin{proof}[Proof of Theorem \ref{main1}]
Case (i). $\alpha<1-k\beta$.
Let $u(.,t)$ be the solution. By rescaling $M_0$ if necessary, we assume
$$a:=\min_{s^n}u(.,0)\le1\le\max_{s^n}u(.,0)=:b.$$
We also
introduce two time-dependent functions
\begin{eqnarray*}
u_1&=&[1-(1-a^{-q})e^{q\gamma t}]^{-\frac{1}{q}},\\
u_2&=&[1-(1-b^{-q})e^{q\gamma t}]^{-\frac{1}{q}},
\end{eqnarray*}
where $q=\alpha+k\beta-1<0$. Both functions $u_1$ and $u_2$ are the solutions of \eqref{NEF1}, hence
by the comparison principle, $u_1(t)\le u(.,t)\le u_2(t)$. That is
$$(a^{-q}-1)e^{q\gamma t}\le u^{-q}-1\le(b^{-q}-1)e^{q\gamma t}.$$
Thus $u$ converges to $1$ exponentially.
By the interpolation and the a priori estimates Lemmas \ref{C0estimate1} and \ref{boundsprincipleradii1}, we see that $\|u(.,t)-1\|_{C^k(S^n)}\to0$ exponentially. Hence $M_t$ converges to the unit sphere centered at the origin.

Case (ii). $\alpha=1-k\beta$.
 By Lemma \ref{GradEstimate3}, we have that $\|\nabla u\|\to0$ exponentially as $t\to\infty$. Hence by the interpolation and the a priori estimates, we can get that $u$ converges exponentially to a constant in the $C^\infty$ topology as $t\to \infty$.
\end{proof}

\begin{proof}
[Proofs of Theorems \ref{main2} and \ref{main3} ]
Recall  Lemma \ref{monotone}. To complete the proof of Theorems \ref{main2} and \ref{main3}, it suffices to show that the solution of \eqref{elliptic eq} is unique.\\
Case1: $\alpha<1-k\beta$.
Let $u_1$,
$u_2$ be two smooth solutions of \eqref{elliptic eq}, i.e.
\[fu_1^{\alpha-1}\sigma_k^{\beta}(\nabla^2u_1+u_1I)=c , fu_2^{\alpha-1}\sigma_k^{\beta}(\nabla^2u_2+u_2I)=c.\]
Suppose $M=\frac{u_1}{u_2}$ attains its maximum at $X_0\in S^n$, then at $x_0$,
\[0=\nabla{\log M}=\frac{\nabla u_1}{u_1}-\frac{\nabla u_2}{u_2},\]
\[0\ge\nabla^2{\log M}=\frac{\nabla^2u_1}{u_1}-\frac{\nabla^2u_2}{u_2}.\]
Hence at $x_0$, we get
\[1=\frac{u_1^{\alpha-1}\sigma_k^{\beta}(\nabla^2u_1+u_1I)}{u_2^{\alpha-1}\sigma_k^{\beta}(\nabla^2u_2+u_2I)}=\frac{u_1^{\alpha-1+k\beta}\sigma_k^{\beta}(\frac{\nabla^2u_1}{u_1}+I)}{u_2^{\alpha-1+k\beta}\sigma_k^{\beta}(\frac{\nabla^2u_2}{u_2}+I)}\le M^{\alpha-1+k\beta}.\]
Since $\alpha<1-k\beta$, $M(x_0)=\max_{S^n} M\le1$. Similarly one can show $\min_{S^n} M\ge1$. Therefore $u_1\equiv u_2$.\\
Case 2: $\alpha=1-k\beta$. We use the same method in \cite{HMS04} to get the solutions for equation\eqref{elliptic eq} differ only by dilation. We omit the proof process here.

Hence we complete the proofs of the Theorems \ref{main2} and \ref{main3}.
\end{proof}

\section{Proof of Theorem \ref{main4} }
In this section we give the proof of Theorem \ref{main4}. The methods for the cases that $f$ being any smooth positive function and $f\equiv1$ are the same , we just consider the case $f\equiv1$. The calculation for the example of the flow \eqref{EF1} is similar to one in \cite{LSW16, LSW17}, we give the brief proof.

By a simple calculation, we obtain the following expressions of the metric, the unit normal, the second fundamental forms and the support function of the hypersurface, according to the radial function $r$
\begin{eqnarray}\label{gij}
g_{ij}=r^2\delta_{ij}+r_ir_j,
\end{eqnarray}
\begin{eqnarray}
n=\frac{rz-{\nabla}r}{\sqrt{r^2+|{\nabla}r|^2}},
\end{eqnarray}
\begin{eqnarray}\label{hbarij}
{\hbar}_{ij}=\frac{-rr_{ij}+2r_ir_j+r^2\delta_{ij}}{\sqrt{r^2+|{\nabla}r|^2}},
\end{eqnarray}
\begin{eqnarray}
u=\frac{r^2}{\sqrt{r^2+|{\nabla}r|^2}}.
\end{eqnarray}

For the convex body $K$ and the hypersurface $M=\partial K$, consider the dual body $K^*$, by \cite{HLYZ16} we know $r^*=\frac{1}{u}$, and
\begin{equation}\label{relationeqs}
K=\frac{\det {\hbar}_{ij}}{\det g_{ij}},
\qquad \frac{1}{\sigma_n^*}=\frac{\det e_{ij}}{\det {\hbar}^*_{ij}},
\qquad\frac{\det e_{ij}}{\det g_{ij}}=\frac{1}{r^{2n-2}(r^2+|{\nabla}r|^2)}.
\end{equation}
Since $r^*=\frac{1}{u}$,
\[\frac{\partial{r^*}}{\partial{t}}=-{r^*}^{2-\alpha}\sigma_k^\beta([{r^*}^{-1}\delta_{ij}+2{r^*}^{-3}r^*_ir^*_j-{r^*}^{-2}r^*_{ij}]).\]
Generally, we consider
\begin{equation}\label{r*t}
\left\{\begin{array}{ll}
\frac{\partial r}{\partial{t}}(x,t)& =-{r}^{2-\alpha}\sigma_k^\beta([{r}^{-1}\delta_{ij}+2{r}^{-3}r_ir_j-{r}^{-2}r_{ij}])(x,t)\\
r(\cdot, 0)&=r_0
\end{array}\right.
\end{equation}
If $\alpha+k\beta-1>0$, that is  $2-\alpha< k\beta+1$. \eqref{ratio} is equivalent to
\begin{equation}\label{ratio*}
R(X(.,t))=\frac{\max_{S^n}r(.,t)}{\min_{S^n}r(.,t)}\to\infty,
 \end{equation}as $t \to T$
for some $T>0$. It is suffice to show that $\min_{S^n}r\to0$ in finite time while $\max_{S^n}r$ remains positive.
By the comparison principle, it suffices to construct a sub-solution $Y(.,t)$
 \begin{equation}
\left\{\begin{array}{ll}
\frac{\partial w}{\partial t}(x,t)&\ge-w^{2-\alpha}\sigma_k^{\beta}([{w}^{-1}\delta_{ij}+2{w}^{-3}w_iw_j-{w}^{-2}w_{ij}])(x,t)\\
w(.,0)&=w_0
\end{array}\right.
\end{equation}such that $\min_{S^n}w(.,t)\to 0$ but $\max_{S^n}w(.,t)$ remains positive, as $t\to T$ for some finite time $T>0$.
\begin{lem}
There is a sub-solution $Y(.,t)$, where $t\in (-1,0)$, to
\begin{equation}\label{a*}
\left\{\begin{array}{ll}
\frac{\partial r}{\partial t}&=-ar^{2-\alpha}\sigma_k^{\beta}([{r}^{-1}\delta_{ij}+2{r}^{-3}r_ir_j-{r}^{-2}r_{ij}])\\
r(.,0)&=r_0
\end{array}\right.
\end{equation}
for a sufficiently large constant $a>0$, such that $\min_{s^n}r(.,t)\to 0$ but $\max_{s^n}r(.,t)$ remains positive, as $t\to0$.
\end{lem}
\begin{proof}
The sub-solution is a family of closed convex hypersurfaces $
\widehat M_t:=Y(S^n, t)$.
Near the origin, let $\widehat M_t$ be the graph of a radial function
\begin{equation}
\psi(\rho,t)=\left\{\begin{array}{llcl}
&-|t|^{\theta}+|t|^{-\theta+\mu\theta}\rho^2 & {\rm{if}}  &\rho<|t|^\theta\\
&-|t|^{\theta}-\frac{1-\mu}{1+\mu}|t|^{{\theta}(1+\mu)}+\frac{2}{1+\mu}\rho^{1+\mu} & {\rm{if }} &|t|^\theta\le\rho\le1
\end{array}
\right.
\end{equation}
where $q=k\beta+1-\widehat\alpha>0$, $\widehat\alpha=2-\alpha$, $\mu=\frac{q\theta-1}{k\beta\theta}$, and $\theta>\frac{1}{q}$ is a constant. It is easy to verify that $\psi$ is strictly convex, and $\psi\in C^{1,1}(B_1(0))$.

If $Y(S^n, t)$ is a sub-solution for some $\widehat\alpha$, it is also a sub-solution for $\alpha'<\widehat\alpha$, so we prove the case when $q=k\beta+1-\widehat\alpha$ is very small.
From the equalities \eqref{gij}, \eqref{hbarij},\eqref{relationeqs} and Newton-Maclaurin inequality, by direct computation, we have

(i) if $0\le\rho\le|t|^\theta$, then
$$r^{\widehat\alpha}\sigma_k^{\beta}([{r}^{-1}\delta_{ij}+2{r}^{-3}r_ir_j-{r}^{-2}r_{ij}])\ge c|t|^{\theta\widehat\alpha}|t|^{k\beta\theta(\mu-1)}=c|t|^{\theta-1},$$
$$|\frac{\partial}{\partial t}Y(p,t)|\le\theta|t|^{\theta-1},$$
 where $p=(x,\psi(|x|,t))$ is a point on the graph of $\psi$.

(ii) if $|t|^\theta\le\rho\le1$, then
$$r^{\widehat\alpha}\sigma_k^{\beta}([{r}^{-1}\delta_{ij}+2{r}^{-3}r_ir_j-{r}^{-2}r_{ij}])\ge c\rho^{\widehat\alpha}\rho^{k\beta(\mu-1)}=c|t|^{\theta-1},$$
$$|\frac{\partial}{\partial t}Y(p,t)|\le\theta|t|^{\theta-1}.$$

 Extending the graph of $\psi$ to a closed convex hypersurface $\widehat M_t $, such that it is $C^{1,1}$ smooth, uniformly convex, rotationally symmetric. Moreover, assume that the ball $B_1(z)$ is contained in the interior of $\widehat M_t$, for all $t\in(-1,0)$, where $z=(0,...,0,10)$ is a point on the $x_{n+1}$-axis.
\end{proof}

For a given $\tau\in(-1,0)$, let $B_1(z)\subset M_0\subset\widehat M_\tau$ , let $M_t$ be a solution to the flow \eqref{a*} with initial data $M_0$. $M_t$ touches the origin at $t=t_0$, for some $t_0\in(\tau,0)$. We assume  $\tau\to 0$ closely.

On the other hand, let $\widetilde X(.,t)$ be the solution to
$$\frac{\partial X}{\partial t}=-ba\widetilde r^{\widehat\alpha}\sigma_k^{\beta}\nu$$
with initial condition $\widetilde X(.,\tau)=\partial B_1(z)$, where $b=2^{\widehat\alpha} sup\{|p|^{\widehat\alpha}: p\in M_t,\tau<t<t_0\}$ and $\widetilde r=|X-z|$. We can choose $\tau$ so small that the ball $B_\frac{1}{2}(z)$ is contained in the interior of $\widetilde X(.,t)$ for all $t\in(\tau,t_0)$.
By the comparison principle, we know that the ball $B_\frac{1}{2}(z)$ is contained in the interior of $M_t$ for all $t\in(\tau,t_0)$. Hence, as $t\to t_0$, we have $\min r(.t)\to0$ and $\max r(.,t)>|z|=10$. Hence \eqref{ratio*} is proved for $M_t$.

For a large constant $a>0$. Making the rescaling $\widetilde M_t=a^{-\frac{1}{q}}M_t$, $\widetilde M_t$ solve the flow \eqref{r*t}. Hence we complete the proof.

\vskip10pt


\end{document}